\documentclass[conference]{IEEEtran}

\usepackage[utf8]{inputenc} 
\usepackage[T1]{fontenc}
\usepackage{url}
\usepackage{ifthen}
\usepackage{cite}
\usepackage[cmex10]{amsmath} 
\usepackage{algorithm}
\usepackage{algpseudocode}

\usepackage{graphicx} 
\usepackage{amssymb}
\usepackage{amsthm}
\usepackage{bbm,bm}
\usepackage{xcolor}
\newtheorem{definition}{Definition}

\newtheorem{theorem}{Theorem}
\newtheorem{lemma}{Lemma}
\newtheorem{remark}{Remark}

\usepackage{comment}

\def\BibTeX{{\rm B\kern-.05em{\sc i\kern-.025em b}\kern-.08em
    T\kern-.1667em\lower.7ex\hbox{E}\kern-.125emX}}

\newcommand{\Exp}{\mathbf{E}}
\newcommand{\Pro}{\mathbf{P}}

\newcommand{\1}{\mathbbm{1}}
\newcommand{\ind}[1]{\1_{\{#1\}}}

\newcommand{\Lap}{\mathrm{Lap}}

\newcommand{\algname}{DP-SUM-CUSUM}
\newcommand{\DeltaT}{\Delta_{\rm max}}

\newcommand{\lixing}[1]{\color{red}{#1 }\color{black}}
\newcommand{\zhang}[1]{\color{brown}{#1 }\color{black}}

\begin{document}
\title{Sequential Change Detection for Multiple
Data \\ Streams  with Differential Privacy} 



 \author{%
   \IEEEauthorblockN{Lixing Zhang\IEEEauthorrefmark{1},
                     Liyan Xie\IEEEauthorrefmark{1},
                        Ruizhi Zhang\IEEEauthorrefmark{2}
                     }
   \IEEEauthorblockA{\IEEEauthorrefmark{1}%
                     Department of Industrial and Systems Engineering, University of Minnesota,
                     \{zhan9503, liyanxie\}@umn.edu }
   \IEEEauthorblockA{\IEEEauthorrefmark{2}%
                     Department of Statistics, University of Georgia, ruizhi.zhang@uga.edu
                    }
   }

\maketitle


\begin{abstract}
Sequential change-point detection seeks to rapidly identify distributional changes in streaming data while controlling false alarms. Existing multi-stream detection methods typically rely on non-private access to raw observations or intermediate statistics, limiting their usage in privacy-sensitive settings. 
    We study sequential change-point detection for multiple data streams under differential privacy constraints. We consider multiple independent streams undergoing a synchronized change at an unknown time and in an unknown subset of streams, and propose \algname, a differentially private detection procedure based on the summation of per-stream CUSUM statistics with calibrated Laplace noise injection. We show that \algname{} satisfies sequential $\varepsilon$-differential privacy and derive bounds on the average run length to false alarm and the worst-case average detection delay, explicitly characterizing the privacy--efficiency tradeoff. A truncation-based extension is also presented to handle distributional shifts with unbounded log-likelihood ratios. Simulations and experiments on an Internet of Things (IoT) botnet dataset validate the proposed approach.
\end{abstract}


\section{Introduction}

Sequential change-point detection aims to rapidly detect distributional changes in streaming data while controlling the false alarm rate, and is a fundamental problem in statistics, signal processing, and information theory \cite{poor-hadj-QCD-book-2008,Siegmund1985,tartakovsky2014sequential,Lai:2001,tutorial_jsait}. It plays a pivotal role in a wide range of real-world tasks, such as health monitoring \cite{balageas2010structural}, misinformation and fake-news detection in social networks \cite{lazer2018science,li2017detecting}, and threat detection \cite{polunchenko2012nearly}.

This work focuses on sequential change detection for multi-stream data. We assume that the distribution changes synchronously in an unknown subset of data streams, and aim to rapidly detect the unknown change-point. Existing multi-stream change-point detection methods typically assume full observability of raw data or channel-level statistics and compute detection statistics directly from these quantities \cite{xie:2013,Wang:2015,liu2019scalable,Spectral-CUSUM2023,chen2022high}. This assumption is increasingly incompatible with privacy requirements in domains such as user monitoring, healthcare, financial transactions, and network event logging \cite{cai2021cost}. 
In these settings, streaming observations often contain sensitive user-level information, and releasing intermediate statistics may leak private personal data. 

In this work, we introduce a privacy-preserving framework for change-point detection in multi-stream settings under differential privacy constraints \cite{dwork2006calibrating}. Our proposed private detection procedures are designed by aggregating evidence across streams while injecting calibrated random noise to ensure sequential $\varepsilon$-differential privacy. Our proposed methods are based on multi-stream CUSUM-type statistics and are computationally efficient for online implementation. We provide a rigorous privacy analysis based on sensitivity bounds for the multi-stream detection statistics. We also derive explicit theoretical guarantees that characterize the tradeoff between privacy and detection performance, quantified by the average run length (ARL) to false alarm and the worst-case average detection delay (WADD). We further extend our framework to settings with unbounded log-likelihood ratios via a truncation strategy. Finally, we demonstrate the effectiveness of the proposed methods through both simulation studies and experiments on a real-world IoT botnet dataset.

The remainder of the paper is organized as follows. 
Section~\ref{sec:problem-setup} introduces the multi-stream change detection model and the notion of differential privacy. Section~\ref{sec:proposed-method} presents \algname, a differentially private multi-stream detection procedure, establishes its privacy guarantees, and provides theoretical analyses of false-alarm control and detection delay. 
Section~\ref{sec:numerical-results} reports numerical results from simulations and real data experiments. Section~\ref{sec:conclusion} concludes the paper.

\subsection{Related Work}

Prior work on sequential multi-stream change detection has studied detection across independent streams \cite{xie:2013,Mei2010,Wang:2015,zou2015efficient,fellouris2016second,chan2017optimal,zhang:2018,liu2019scalable,enikeeva2019high,cao2019sketching,zhang2022robust,cao2023adaptive} and high-dimensional correlated data streams \cite{jiao2018subspace,zou2009multivariate,yan2018real,keshavarz2020sequential,qiu2020big,hewapathirana2020change,xie2020subspace,Spectral-CUSUM2023,sha2022quickest,chen2022high,chen2022monitoring,gosmann2022sequential}. However, all of these works compute the detection statistics directly from raw observations and pass them to the decision process without privacy safeguards. 
Recent work has developed private sequential detection methods under differential privacy, spanning single-stream change estimation \cite{cummings2018differentially,cummings2020privately}, graphical models \cite{DP-CBM}, local differential privacy \cite{berrett2021locally,zhang2026localdp,yadav2026locally}, and distributed or multi-stream settings \cite{kurt2022online}.

A recent work \cite{dp-cusum-2025} is among the first to study sequential change detection under an $\varepsilon$-differential privacy constraint for single-stream data, and characterizes the impact of the privacy parameter $\varepsilon$ on key performance metrics such as the average run length and the worst-case average detection delay. In this work, we extend the core analysis of \cite{dp-cusum-2025} to the multi-stream setting, addressing the challenges arising from aggregating local statistics across streams in a privacy-preserving manner.


\section{Problem Setup and Preliminaries}
\label{sec:problem-setup}

We consider $K$ independent data streams $\{X^k_t\}_{t\ge1}$, $k=1,\dots,K$. Initially, the $X^{k}_t$ are distributed according to the density $f_{0,k}$ for $k = 1,\ldots, K.$ At some unknown time $\tau$, an unusual event occurs and affects an unknown subset of data streams in
the sense that if the $k^{th}$ data stream is affected, the density function of its local observations $X^k_t$
changes from $f_{0,k}$ to $f_{1,k}$ after time $\tau.$  Here, $f_{0,k}$ and $f_{1,k}$ denote the pre-change and post-change distributions for stream $k$ and are assumed to be known. However, we assume the set of affected data streams and the number of affected data streams $1\le m\le K$ is {\it unknown}.




To define differential privacy in the multi-stream setting, we first specify
a notion of neighboring data streams.

\begin{definition}[Neighboring data streams] 
\label{def:neighbor}
Two data streams $\mathbf X_{(1:n)}:=\{X^k_t\}_{k\in [K], t\in [n]}$ and $\tilde{\mathbf X}_{(1:n)}:=\{\tilde{X}^k_{t}\}_{k\in [K], t\in [n]}$ are called neighboring streams if they only differ at a single time step, say $t_0$, and a single data stream $k_0$:
\[
X^{k_0}_{t_0} \neq \tilde{X}^{k_0}_{t_0}; \, X^{k}_{t_0} = \tilde{X}^{k}_{t_0} , \forall k\neq k_0;\, X^k_{t} = \tilde{X}^k_{t},\forall k, \forall t \neq t_0.
\] 
\end{definition}
Based on the neighboring relation above, we now define differential privacy
for sequential detection rules by requiring privacy guarantees at every
possible stopping time.

\begin{definition}[$\varepsilon$-DP sequential detection on multi-streams]
    \label{def:edp}
    A randomized sequential detection procedure with stopping time $\mathcal{T}$ is said to be $\varepsilon$-differentially private ($\varepsilon$-DP) if for any pair of neighboring data streams $\mathbf X_{(1:n)}$ and $\tilde{\mathbf X}_{(1:n)}$ in the sense
    of Definition~\ref{def:neighbor}
    and for any $t\geq 1$, 
    \begin{equation}\label{eq:dp}       \Pro_\mathcal{T}\big(\mathcal{T}=n |
    \mathbf X_{(1:n)}\big)
\le
    e^{\varepsilon}  \Pro_\mathcal{T}\big(\mathcal{T}=n | \tilde{\mathbf X}_{(1:n)}\big).
    \end{equation}
    Here, $\Pro_\mathcal{T}$ denotes the probability measure induced by the randomness of the stopping time $\mathcal T$.
\end{definition}
Definition~\ref{def:edp} means that altering any single observation only slightly affects the distribution
of the randomized stopping time $\mathcal T$ so that one cannot easily infer individual data values from the detection output.
Here, a randomized sequential detection rule is a stopping time $\mathcal T,$ where the decision $\{\mathcal T=n\}$ is only based on observations up to time $n,$ i.e., $\{X^k_t\}_{k\in [K], t\in[n]}$ and the additional random noises added to the procedure.
We use $\Pro_\infty,\Exp_\infty$ to denote joint probability and expectation of the data under the no-change regime and the added random noise, and $\Pro^{(k_1,k_2,\cdots, k_m)}_\tau,\Exp^{(k_1,k_2,\cdots, k_m)}_\tau$ when the change-point occurs at $\tau$ and the density of the observation $X_t^k$ changes from $f_{0,k}$ to $f_{1,k}$ only at the $k^{th}$ data stream for $k=k_1,k_2,\cdots,k_m$ and there are no changes at other data streams. As a special case, $\Pro^{(k_1,k_2,\cdots, k_m)}_0,\Exp^{(k_1,k_2,\cdots, k_m)}_0$ corresponds to a change occurring at time $\tau=0$.

We consider two common performance metrics for a randomized stopping time $\mathcal T$: (i) {\it Average run length (ARL)} to false alarm measures the expected time to a false alarm when no change is present: $\mathrm{ARL}(\mathcal T)=\Exp_\infty[\mathcal T]$; (ii) {\it Worst-case average detection delay} (WADD) over all possible change-points and pre-change data \cite{Lorden1971}:
\begin{align*}
&\mathrm{WADD}^{(k_1,k_2,\cdots, k_m)}(\mathcal T)\\
    =&\sup_{\tau\ge 0}\text{esssup } \Exp_\tau^{(k_1,k_2,\cdots, k_m)}\!\big[(\mathcal T-\tau)^+\big|
    \{X^k_t\}_{k\in[K],t\in [\tau]}\big].
    \end{align*}
Here, $\text{esssup}$ denotes the essential supremum of the conditional expected delay with respect to the pre-change histories.
We aim to develop an $\varepsilon$-DP detection procedure that satisfies pre-specified ARL constraints while achieving small detection delay, and to quantify the trade-off between detection delay and the privacy budget $\varepsilon$.

As a non-private benchmark, the classical CUSUM procedure is an optimal method for change-point detection without privacy guarantees and serves as a building block for our proposed $\varepsilon$-DP procedures.
For stream $k$, we define the log likelihood ratio (LLR) 
\begin{equation}\label{eq:llr}
\ell_k(x)=\log\frac{f_{1,k}(x)}{f_{0,k}(x)},
\end{equation}
and Kullback-Leibler information $I_{0,k}:=\Exp_0[\ell_k(X_1^k)]=\int f_{1,k}(x) \ell_k(x) dx $.
The CUSUM statistic for stream $k$ is defined as \cite{page-biometrica-1954} 
\begin{equation}\label{eq:cusum}
\begin{aligned}
S_{t}^k=\max\{0,S_{t-1}^k+\ell_k(X_t^k)\}, t\ge 1, \text{ with } S_{0}^k=0,
\end{aligned}
\end{equation}
and the associated CUSUM stopping time is the first time $S_{t}^k$ exceeds a pre-set threshold.  

\section{Proposed Method: \algname}
\label{sec:proposed-method}

In this section, we present our proposed $\varepsilon$-DP detection procedure for multi-stream data, which we call \algname{}. For simplicity, we first focus on the setting where the log-likelihood ratio $\ell_k$ as defined in Eq.~\eqref{eq:llr} is bounded for all data streams. This case allows us to introduce the main ideas underlying our
privacy-preserving procedure and to derive
clean performance guarantees. Moreover, we extend our method to handle unbounded log-likelihood ratios via a truncation strategy, which allows the privacy and ARL/WADD arguments to extend with minor modifications.

Our proposed method builds upon the classical CUSUM procedure and is constructed as follows. We first define the per-stream sensitivity 
\[
\Delta_k := \sup_{x,y}|\ell_k(x)-\ell_k(y)|, \, k =1,\ldots,K,
\]
and the global sensitivity $\DeltaT:=\max_{1\le k\le K}\Delta_k$.  
For each stream $k$, we maintain a classical CUSUM statistic $S_t^k$ as in Eq.~\eqref{eq:cusum} and then combine the per-stream CUSUM statistics through summation, similar to \cite{Mei2010}: 
\begin{equation}\label{eq:sum-cusum}
U_t = \sum_{k=1}^K S_{t}^k.    
\end{equation}

To enforce differential privacy, we add independent noise to both the detection statistic and the threshold. Specifically, let $Z_t$ and $W$ be independent zero-mean Laplace random variables with distribution $\Lap(2\Delta_{\rm max}/\varepsilon)$. The stopping time of our proposed procedure is
\begin{equation}\label{eq:dp-cusum}
\mathcal T(b)=\inf\{t\geq 1: U_t+Z_t\ge b+W\}. 
\end{equation}
Here, the Laplace noise $Z_t$ protects the privacy of each individual data, while $W$ prevents information leakage through repeated or
adaptive comparisons over time. The full procedure is summarized in Algorithm~\ref{alg:dp-sum-cusum}.

\begin{algorithm}[b]
\caption{\algname}
\label{alg:dp-sum-cusum}
\begin{algorithmic}[1]
\Statex \textbf{Input:} Threshold $b>0$, privacy parameter $\varepsilon>0$, global sensitivity $\DeltaT$.
\Statex \textbf{Output:} Stopping time $\mathcal T(b)$.
\State \textbf{Initialize:} $t=0$, $Z_0=0$, $U_0=0$, $S_{0}^k=0, \forall k$.
\State Sample $W \sim \Lap(2\Delta_{\rm max}/\varepsilon)$.
\While{$U_t + Z_t \ge b + W$}
    \State $t \leftarrow t+1$.
    \For{$k=1,\dots,K$}
        \State $S_{t}^k \leftarrow \max\{0, S_{t-1}^k + \ell_k(X_t^k)\}.$ 
    \EndFor
    \State Aggregate statistics: $U_t \leftarrow \sum_{k=1}^K S_{t}^k.$
    \State Sample $Z_t \sim \Lap(2\Delta_{\rm max}/\varepsilon)$.
\EndWhile
\State Output stopping time $\mathcal T(b)=t$ and declare a change has occured before time $t$.
\end{algorithmic}
\end{algorithm}

We first prove the proposed \algname\ procedure satisfies the $\varepsilon$-DP requirement as defined in Definition~\ref{def:edp}. 
%






\begin{theorem}[Sequential $\varepsilon$-DP]
\label{thm:T1-EDP}
The procedure $\mathcal T(b)$ in Eq.~\eqref{eq:dp-cusum} is sequentially $\varepsilon$-DP: for all $n\ge1$ and neighboring data streams $\mathbf X_{(1:n)}$, $\tilde{\mathbf X}_{(1:n)}$ (in the sense of Definition~\ref{def:neighbor}),
\[
\Pro_{\mathcal T}\big(\mathcal T=n\mid \mathbf X_{(1:n)}\big)\le e^{\varepsilon}\Pro_{\mathcal T}\big(\mathcal T=n\mid \tilde{\mathbf X}_{(1:n)}\big).
\]
\end{theorem}

\begin{proof}[Proof]
First, it has been shown in \cite[Lemma 3]{dp-cusum-2025} that for stream $k$, if $X^k_{(1:n)}$ and $\tilde X^k_{(1:n)}$ differ in at most one time index, then their corresponding CUSUM statistics $S_t^k$ and $\tilde S_t^k$ satisfies $\big|S_{t}^k-\tilde S_{t}^k \big| \le \Delta_k$, $\forall t\le n$. Then we have the corresponding summation of the per-stream CUSUM statistic satisfies
\[
|U_t-\tilde U_t| \le \DeltaT,\, \forall 1\le t \le n.
\]
Following the proof of \cite[Theorem 1]{dp-cusum-2025}, one can show that the stopping time $\mathcal T$ in Eq.~\eqref{eq:dp-cusum} is $\varepsilon$-DP.
\end{proof}


We then analyze the false-alarm and detection-delay performance of \algname{}. The following results provide the lower bound on ARL and the upper bound on WADD.
In particular, the ARL bound shows that false alarms can still be controlled exponentially in the threshold, and WADD scales on the order of $b/I_{\rm tot}$, where $I_{\rm tol}$ denotes the total Kullback–Leibler information across the streams affected by the change. Together, these results characterize the fundamental tradeoff between privacy budget $\varepsilon$ and detection efficiency.
Compared to \cite{dp-cusum-2025}, the main technical difficulty here is that we work with an aggregated multi-stream statistic with an unknown subset of affected data streams, so both the post-change drift and the DP sensitivity need to be controlled at the aggregate level.
    
\begin{theorem}[ARL of \algname]\label{thm:ARL-T1}
For $b>K+1$, we have
\begin{equation}\label{eq:arl-multi1}
\Exp_\infty[\mathcal T(b)]\ge \frac{1}{16} e^{h(\varepsilon,\DeltaT)b - (K+1)}\big(\frac{K+1}{b+K+1}\big)^{K+1},
\end{equation}
where $h(\varepsilon,\DeltaT)=\min\{\varepsilon/(2\Delta_{\rm max}),1\}$.
\end{theorem}
    
\begin{proof}[Proof]
We first condition on $W=w$ and similar to the proof of \cite[Theorem 2]{dp-cusum-2025}, we have $\forall x>0$, $\lambda\in (0, h(\varepsilon,\DeltaT))$,
\begin{equation}\label{eq:arl-proof1}
\begin{aligned}
& \Exp_\infty[\mathcal T(b) | W=w] \ge x \Pro_\infty(\mathcal T(b)\ge x | W=w) \\
& \ge x(1-\sum_{n=1}^{\lfloor x \rfloor} e^{-\lambda (b+w)}\Exp_\infty[e^{\lambda Z_n}] \prod_{k=1}^K \Exp_\infty[e^{\lambda S_{n}^k}] )  \\ 
& \ge x\bigg(1-x e^{-\lambda (b+w)} \frac{1}{1-4\DeltaT^2\lambda^2/\varepsilon^2} (\frac{1}{1-\lambda})^K\bigg),
\end{aligned}    
\end{equation}
where the last inequality is due to $\Exp_\infty[e^{\lambda S_{n}^k}]\le 1/(1-\lambda)$ (by \cite[Corollary 3]{dp-cusum-2025}) and $\Exp_\infty[e^{\lambda Z_n}]=1/(1-4\DeltaT^2\lambda^2/\varepsilon^2)$. By choosing $x$ that maximizes the right-hand-side of \eqref{eq:arl-proof1}, we have 
\begin{equation}\label{eq:arl-proof2}
\Exp_\infty[\mathcal T(b) | W=w]\geq \frac{e^{\lambda(b+w)}}{4}(1-\lambda)^K(1-\frac{2\DeltaT}{\varepsilon}\lambda).    
\end{equation}
We then consider the following two cases. If $\varepsilon\le2\DeltaT$, the right-hand-side (RHS) in \eqref{eq:arl-proof2} is lower bounded by $\frac{e^{\lambda(b+w)}}{4}(1-\frac{2\Delta_{\rm max}\lambda}{\varepsilon})^{K+1}$, which is maximized at $\lambda^*=\frac{\varepsilon}{2\Delta_{\rm max}}-\frac{K+1}{b+w}$ if $w>\frac{2(K+1)\Delta_{\rm max}}{\varepsilon}$, thus we have
\begin{align*}
        &\Exp_\infty[\mathcal T(b)]=\int_{-\infty}^{\infty}\Exp_\infty[\mathcal T(b) | W=w]f_W(w)dw\\
        &\ge \int_{\frac{2(K+1)\Delta_{\rm max}}{\varepsilon}}^{\infty} \!\!\! \frac{e^{\frac{\varepsilon(b+w)}{2\Delta_{\rm max}}-(K+1)}}{4}\bigg[\!\frac{2\Delta_{\rm max}(K+1)}{(b+w)\varepsilon}\!\bigg]^{K+1}\!\!f_W(w)dw\\
        &\geq \frac18 e^{\frac{\varepsilon}{2\Delta_{\rm max}}b-(K+1)}(\frac{K+1}{b+K+1})^K.
    \end{align*}
If $\varepsilon>2\DeltaT$, the RHS of \eqref{eq:arl-proof2} is lower bounded by $e^{\lambda(b+w)}(1-\lambda)^{K+1}$, which is maximized at $\lambda^*=1-\frac{K+1}{b+w}$. Then we have 
\begin{align*}
        &\Exp_\infty[\mathcal T(b)]=\int_{-\infty}^{\infty}\Exp_\infty[\mathcal T(b)| W=w]f_W(w)dw\\
        &\ge \int_{0}^{\infty} \frac{e^{b+w-(K+1)}}{4}(\frac{K+1}{b+w})^{K+1}\frac{\varepsilon}{4\Delta_{\rm max}}e^{-\frac{\varepsilon w}{2\DeltaT}}dw\\
        &\ge \frac14\frac{\varepsilon e^{b-(K+1)}}{4\Delta_{\rm max}}(\frac{K+1}{b+1})^{K+1}\int_{0}^{1}e^{-\frac{\varepsilon}{2\Delta_{\rm max}}w}dw\\
        &
        \ge \frac{e^{b-(K+1)}}{16}(\frac{K+1}{b+K+1})^{K+1}.
    \end{align*}
Combining the two cases together completes the proof.   
\end{proof}

\begin{theorem}[WADD of \algname]\label{thm:wadd-T1}
For any $b>0$, we have
\begin{equation}\label{eq:T1-wadd}
\mathrm{WADD}^{(k_1,k_2,\cdots, k_m)}(\mathcal T(b)) 
\le 
\frac{b}{I_\mathrm{tot}} + 
\frac{4\DeltaT}{\varepsilon I_\mathrm{tot}^{3/2}}\sqrt{b}+C,
\end{equation}
where $I_{\mathrm{tot}}:=\sum_{i=1}^m I_{0,k_i}$ is the total information number of those affected data streams and $C$ is a constant depending on $(\varepsilon, \Delta_{\rm max}, I_\mathrm{tot})$ but not on $b$.
\end{theorem}

\begin{proof}[Proof]
Following the proof of \cite[Lemma 2]{dp-cusum-2025}, we have the worst-case delay $\mathrm{WADD}^{(k_1,k_2,\cdots, k_m)}(\mathcal T(b)) \le \Exp_0^{(k_1,k_2,\cdots, k_m)}(\mathcal T(b))$, thus we only need to upper bound $\Exp_0^{(k_1,k_2,\cdots, k_m)}(\mathcal T(b))$. Without loss of generality, we assume that only the first $m$ data streams are affected, and thus we omit the $(k_1,k_2,\cdots,k_m)$ in this proof for simplification. To prove Eq.~\eqref{eq:T1-wadd}, we first derive an upper bound on the conditional detection delay given $W=w$, and then compute the expectation of this bound over the distribution of $W$. Specifically, denote $\tilde b = b+w$, we define a new process $U_t':=\sum_{i=1}^t\sum_{k=1}^m  \ell_k(X_i^k)$ and the new stopping time $\nu(\tilde b)=\inf\{t\ge1: U_t'+Z_t\ge \tilde b\}$. Since each $S_{t}^k \ge 0,\forall k,t$, we have $U_t' \le U_t$ in Eq.~\eqref{eq:sum-cusum} and thus $\Exp_0[\mathcal T(b)| W=w]\le \Exp_0[\nu(\tilde b)].$
Then, we just need to upper bound $\Exp_0[\nu(\tilde b)]$. For simplicity, we use $\nu$ to denote the stopping time $\nu(\tilde b)$ in the following and omit the conditioning on $W=w$ in the following proofs.

First, it is easy to show that $\Exp_0[\nu]$ is bounded, similar to Step 1 in the proof of \cite[Theorem 3]{dp-cusum-2025}. Second, by Wald’s equation, as $\Exp_0[\sum_{k=1}^m  \ell_k(X_1^k)]=I_{\mathrm{tot}}$, we can write
\begin{equation} \label{eq:wald_T1}
\Exp_0[\nu]
=\frac{\Exp_0[U'_\nu]}{I_\mathrm{tot}}
=\frac{\tilde b+\Exp_0[U'_{\nu}+Z_{\nu}-\tilde b]+\Exp_0[-Z_{\nu}]}{I_\mathrm{tot}}.
\end{equation}
We define $Z_0=0$. Note that by definition of $\nu$, $U'_{\nu-1}+Z_{\nu-1}<\tilde b$ and $U'_\nu+Z_\nu-\tilde b\ge0$. Then we have $U'_\nu+Z_\nu-\tilde b=U'_{\nu-1}+\sum_{k=1}^{m}\ell_k(X_\nu^k)+Z_{\nu}-\tilde b+Z_{\nu-1}-Z_{\nu-1}\leq m\Delta_{\rm max}+Z_{\nu}-Z_{\nu-1}$.
Substituting into \eqref{eq:wald_T1} yields

\begin{align}
    \label{eq:wald_T2}
\Exp_0[\nu]&\leq \frac{\tilde{b}+m\Delta_{\rm max}+\Exp_0[Z_\nu-Z_{\nu-1}]+ \Exp_0[-Z_{\nu}]}{I_\mathrm{tot}}\\
&= \frac{\tilde{b}+m\Delta_{\rm max}+\Exp_0[-Z_{\nu-1}]}{I_\mathrm{tot}}\\
&\le \frac{\tilde{b}+m\Delta_{\rm max}+\Exp_0[|Z_{\nu-1}|]}{I_\mathrm{tot}}.  
\end{align}
Third, we derive an upper bound for $\Exp_0[|Z_{\nu-1}|]$. Let $m_1=\lfloor\frac{2\tilde{b}}{I_{\rm tot}}\rfloor+1$, we have 
\begin{align*}
&\Exp_0[|Z_{\nu-1}|]=\sum_{i=1}^\infty\Exp_0[|Z_i\ind{\nu=i+1}|]\\
    &\leq  \underbrace{\sum_{i=1}^{m_1}\Exp_0[|Z_i\ind{\nu=i+1}|]}_{\text{Part 1}}+\underbrace{\sum_{i=m_1+1}^{\infty}\Exp_0[|Z_i|\ind{U_i+Z_i<\tilde b}]}_{\text{Part 2}}.
\end{align*}
By Cauchy-Schwarz inequality, we can show $\text{Part 1} \le \sum_{i=1}^{m_1}(E[|Z_i|^2])^\frac{1}{2}(\Pro_0(\nu=i+1))^\frac{1}{2} \le 2\sqrt{2}\frac{\Delta_{\rm max}}{\varepsilon}\sqrt{m_1}$. For Part 2, for any $i\geq m_1+1$ and $\lambda\in(0,\frac{\varepsilon}{2\DeltaT})$,
\begin{align*}
    &\Exp_0[|Z_{i}|\mathbf{1}_{(U_i+Z_i< \tilde{b})}]\leq (\Exp_0[|Z_{i}|^2)^\frac{1}{2}(\Pro_0(U_i+Z_i< \tilde{b}))^\frac{1}{2}\\
    &\leq 2\sqrt{2}\frac{\Delta_{\rm max}}{\varepsilon}\big(\frac{\Exp_0[e^{-\lambda(\sum_{j=1}^i\sum_{k=1}^{m} \ell_k(X^k_j)+Z_i)}]}{e^{-\lambda \tilde{b}}}\big)^{1/2}\\
    &\leq 2\sqrt{2}\frac{\DeltaT}{\varepsilon}\big(\frac{e^{(\frac{\lambda^2m\DeltaT^2}{8}-I_{\rm tot}\lambda)i+\lambda\tilde{b}}}{1-4\DeltaT^2\lambda^2/\varepsilon^2}\big)^{1/2}.
\end{align*}
Taking $\lambda=\min\{\frac{4I_\mathrm{tot}}{m\Delta_{\rm max}^2}, \frac{\varepsilon}{2\sqrt{2}\Delta_{\rm max}}\}$ in the right-hand-side of the above inequality and take the summation from $i=m_1+1$ to infinity, we have $\text{Part 2}\leq \frac{8\Delta_{\rm max}}{\varepsilon(I_{\rm tot}\lambda-\lambda^2m\Delta_{\rm max}^2/8)}$.
Finally, we substitute the upper bound for $\Exp_0[|Z_{\nu-1}|]$ into Eq.~\eqref{eq:wald_T2}, and take the expectation over $W$ to complete the proof.
\end{proof}

By Theorem~\ref{thm:ARL-T1}, the ARL constraint $\Exp_\infty[\mathcal T(b_{\gamma})]\ge \gamma$ can be satisfied in the asymptotic regime  $\gamma\to\infty$ by choosing $b=b_\gamma$ such that $ \frac{1}{16}e^{h(\varepsilon,\Delta_{\rm max})b_\gamma - (K+1)}(\frac{K+1}{b_\gamma+K+1})^{K+1}=\gamma$. 
This yields
\begin{align*}
b_{\gamma} 
    &=\frac{\log\gamma}{h(\varepsilon,\Delta_{\rm max})}\bigl(1+o(1)\bigr).
\end{align*}
Combining this choice with Theorem~\ref{thm:wadd-T1}, we obtain
\[
\mathrm{WADD}(\mathcal T(b_{\gamma})) \le \frac{\log \gamma}{h(\varepsilon,\Delta_{\rm max})I_{\rm tot}}(1+o(1)).
\]

This shows that, as in many differentially private sequential procedures, stronger privacy protection generally comes at the cost of increased detection delay \cite{berrett2021locally,cummings2020privately}. Therefore, the proposed method is most well-suited to privacy-sensitive monitoring applications where formal privacy guarantees are required and a slight loss in detection efficiency is acceptable.

We extend the \algname{} procedure to the case when LLR $\ell_k$ is unbounded for some streams via the truncation strategy described in the following remark. We note that truncation is necessary here to ensure finite sensitivity, and hence differential privacy, when the log-likelihood ratio is unbounded. At the same time, truncation can limit the contribution of extreme observations and thereby reduce the information available for detection, again reflecting the inherent trade-off between privacy protection and detection efficiency. In practice, the truncation level can be chosen to keep the truncated information numbers sufficiently large, so that the detector can still accumulate post-change evidence effectively and maintain meaningful detection power.
\begin{remark}[Unbounded LLR]\label{sec:truncation}
Let $\mathcal U \subset [K]$ be the set of streams with unbounded LLR. For each stream $k\in\mathcal U$, we define its truncated LLR as
\begin{equation}
    \label{eq:truncate}
    \tilde\ell_k(x) := \min\{|\ell_k(x)|,\Delta'/2\} \mathrm{sign}(\ell_k(x)),
\end{equation}
where $\Delta'$ is a fixed constant and $\mathrm{sign}$ denotes the sign function. Since the sensitivity of truncated LLR $\tilde \ell_k$ is $\Delta'$, we can then define the new global sensitivity as $\DeltaT'=\max\{\max_{k\notin\mathcal U} \Delta_k, \Delta'\}<\infty$. Then we can construct the \algname\ procedure similarly by replacing
$\ell_k(\cdot)$ with $\tilde\ell_k(\cdot)$ in Line 6 of Algorithm~\ref{alg:dp-sum-cusum}. Specifically, we define $\tilde S_t^k = \max\{0,\tilde S_{t-1}^k+\tilde \ell_k(X_t^k)\}$ for $k\in\mathcal U$, and let $\tilde U_t=\sum_{k\in\mathcal U} \tilde S_t^k + \sum_{k\notin\mathcal U} S_t^k$.  
Then the resulting stopping time $\tilde {\mathcal T}(b)$ is
\[
\tilde {\mathcal T}(b):= \inf\{t\geq 1: \tilde U_t+ Z_t\ge b+W\},
\]
where $Z_t,W\sim \Lap(2\Delta'_{\rm max}/\varepsilon)$.
Since the global sensitivity $\DeltaT'$ is bounded after such truncation, it follows from the proof of Theorem~\ref{thm:T1-EDP} that $\tilde T$ is still sequentially $\varepsilon$-DP.

In order to ensure effective detection, in practice, the truncation parameter $\Delta'$ is chosen such that $\forall k \in \mathcal U$,
\begin{equation}\label{ass:finite-KL}
\tilde I_{0,k} := \Exp_0[\tilde\ell_k(X_1^k)]>0, \text{ and }
\tilde I_{1,k} := -\Exp_\infty[\tilde\ell_k(X_1^k)]>0.    
\end{equation} 
Such a choice always exists for sufficiently large $\Delta'$, and under this condition, the proofs of Theorem~\ref{thm:ARL-T1} and Theorem~\ref{thm:wadd-T1} still hold with minimum modifications, and we can obtain similar performance guarantees.
\end{remark}



\section{Numerical Results}
\label{sec:numerical-results}

In this section, we evaluate the proposed \algname\ procedure using synthetic simulations and a real-data example, which demonstrates the practical applicability of our approach in privacy-sensitive
multi-stream monitoring tasks.

\subsection{Simulation Results}\label{sec:simulation}

We first simulate a Laplace mean-shift setting with $K=5$ independent data streams. For each stream, the pre-change
distribution is $\Lap(0,1)$ and the post-change distribution
is $\Lap(0.2,1)$. This setting yields bounded log-likelihood ratios, allowing the original \algname\ procedure to be applied without truncation. We compare the performance of the proposed \algname\ procedure with
the non-private classical SUM-CUSUM as a baseline \cite{Mei2010}. For each value of the detection threshold, we simulate the ARL and expected detection
delay with 10,000 independent trials.  
As shown in Fig.~\ref{fig:simulation} (a), for a fixed ARL level, the detection delays
of \algname\ (under $\varepsilon=0.2,0.4$) are slightly higher than that of the non-private SUM-CUSUM, reflecting the cost of enforcing differential privacy.
Nevertheless, the gap remains moderate, and the \algname\ curves
closely track the baseline, especially for larger values of the privacy
budget~$\varepsilon$.

\begin{figure}[t!]
    \centering
    \setlength{\tabcolsep}{0.1pt}
    \begin{tabular}{cc}
\includegraphics[width=0.49\linewidth]{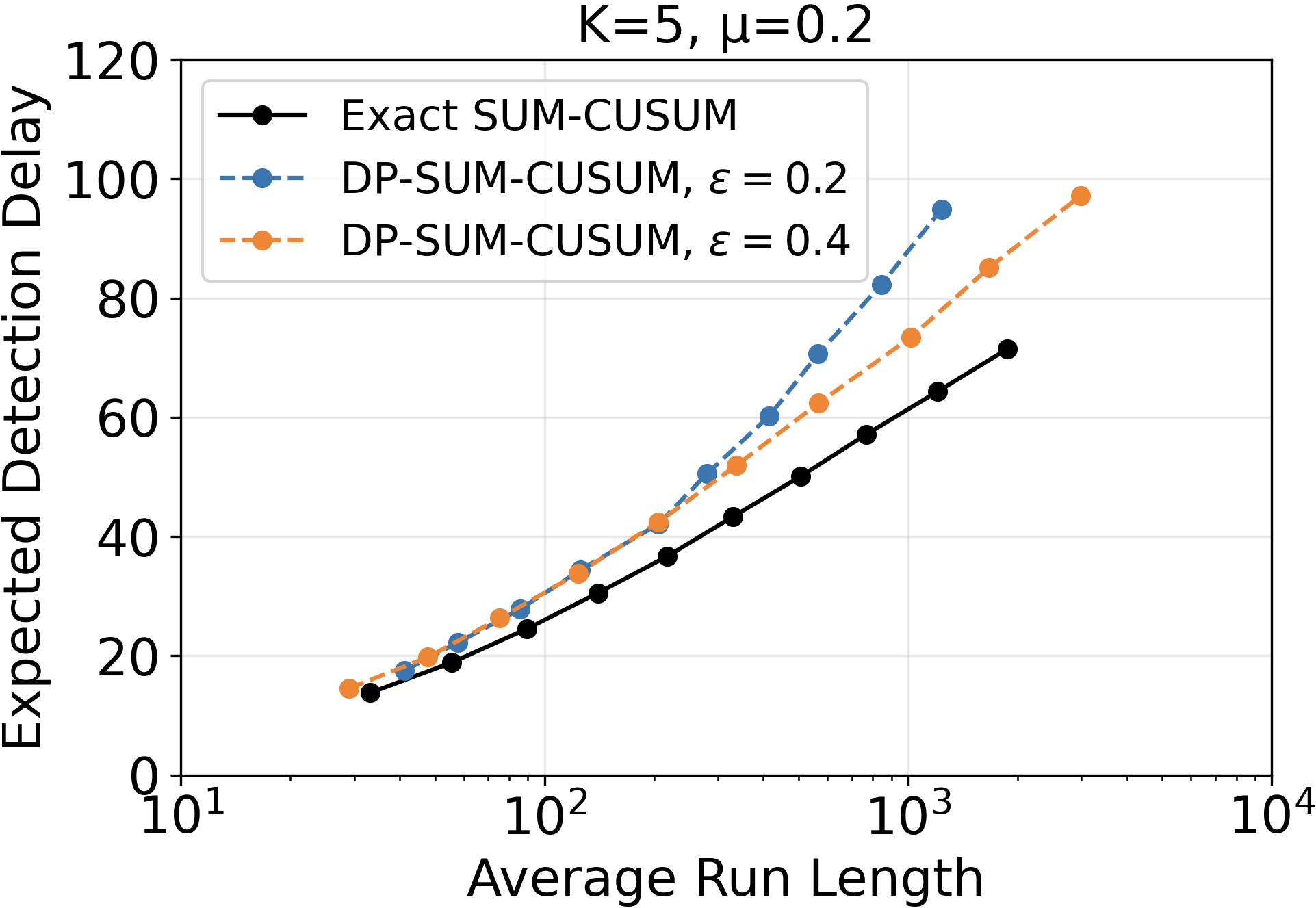} & \includegraphics[width=0.49\linewidth]{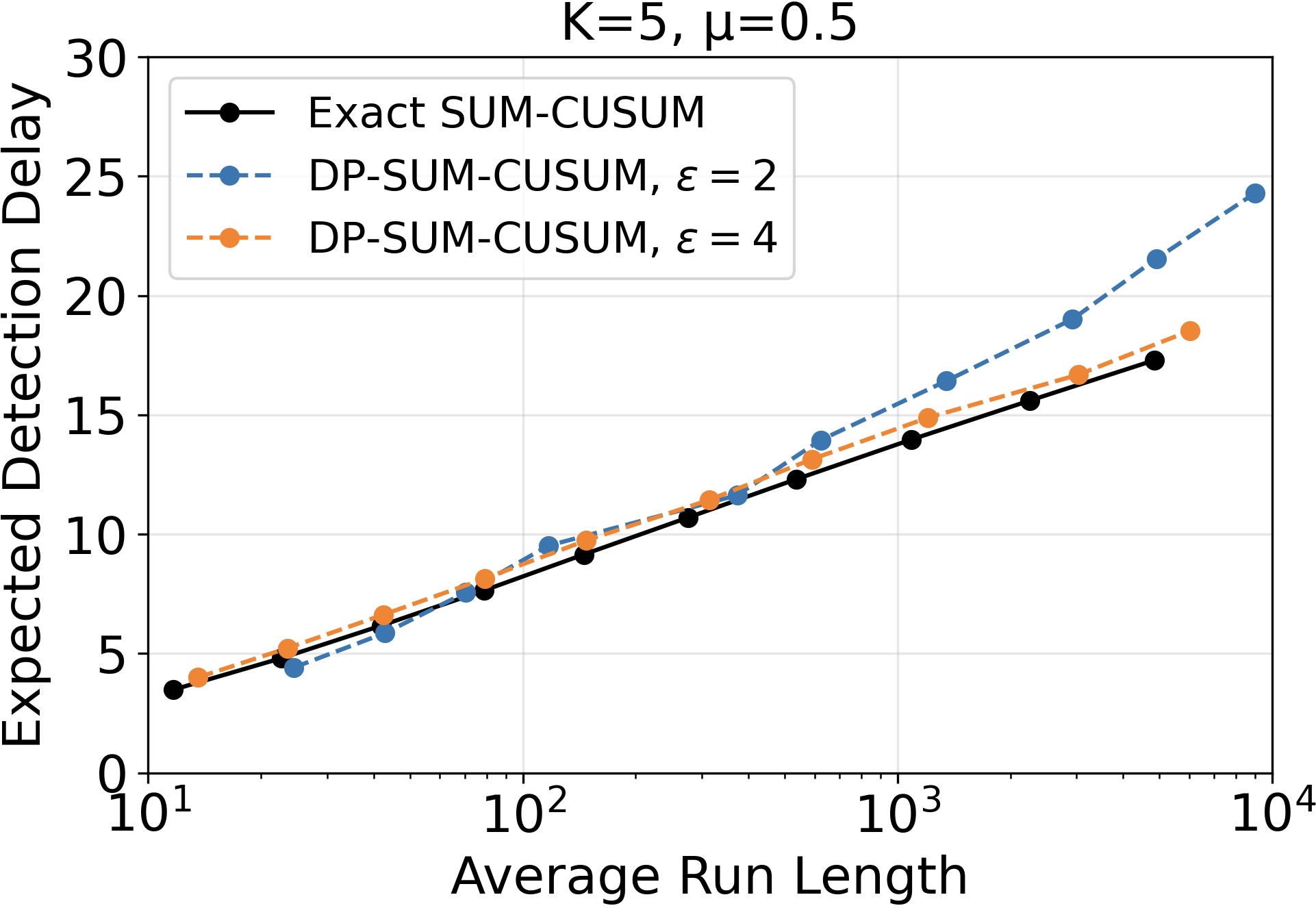}\\
(a) $\Lap(0,1) \xrightarrow{} \Lap(0.2,1)$  & (b) $\mathcal{N}(0,1) \xrightarrow{} \mathcal{N}(0.5,1)$
    \end{tabular}
    \caption{Delay–ARL tradeoff curves comparing \algname\ with the non-private SUM-CUSUM for $K=5$ data streams under (a) Laplace mean-shift $\Lap(0,1)\to\Lap(0.2,1)$ and (b) Gaussian mean-shift $\mathcal{N}(0,1)\to\mathcal{N}(0.5,1)$, where truncation is applied.
    }
    \label{fig:simulation}
\end{figure}

    


We then consider a Gaussian mean-shift setting with $K=5$ streams, where the pre-change distribution is $\mathcal{N}(0,1)$ and the post-change distribution is
$\mathcal{N}(0.5,1)$ for all streams. Since the corresponding log-likelihood ratios are unbounded, we apply the truncation method described in
Remark~\ref{sec:truncation} with truncation parameter $\Delta'=2.5$ that guarantees Eq.~\eqref{ass:finite-KL} is satisfied.
Fig.~\ref{fig:simulation} (b) reports the ARL and detection delay under this Gaussian mean-shift setting averaged over 10,000 independent trials. Despite the use of truncated log-likelihood ratios, the proposed \algname\ procedure maintains a similar ARL–Delay tradeoff structure as in the bounded case. In particular, for
larger values of~$\varepsilon$, the \algname\ curves remain close to the
non-private baseline, demonstrating that truncation does not
substantially degrade detection performance in this regime.

\subsection{Real Data Example}

We evaluate the proposed method on a public Internet of Things (IoT) botnet dataset containing nine heterogeneous consumer devices, including doorbells, thermostats, security cameras, and smart plugs \cite{meidan2018n}. Treating each device as an independent information source yields a multi-stream change-detection problem with $K=9$ streams. 
For each device, the raw observations are time-ordered vectors with $115$ numeric features that summarize the statistics of the flow and packet level. We apply principal component analysis to each stream to reduce the dimensionality to five, and standardize the retained components to have zero mean and unit variance. We treat the benign traffic traces data as pre-change and the data generated under \texttt{junk} attack as post-change. 
 We use historical data collected under benign traffic and previous \texttt{junk} attacks to estimate the pre- and post-change data distributions using Gaussian models. Since the resulting log-likelihood ratios are potentially unbounded, we employ the truncation-based variant of \algname\ described in Remark~\ref{sec:truncation} to ensure finite sensitivity and $\varepsilon$-differential privacy.

\begin{figure}[t!]
        \centering
        \vspace{-0.1in}
      \includegraphics[width=0.9\linewidth]{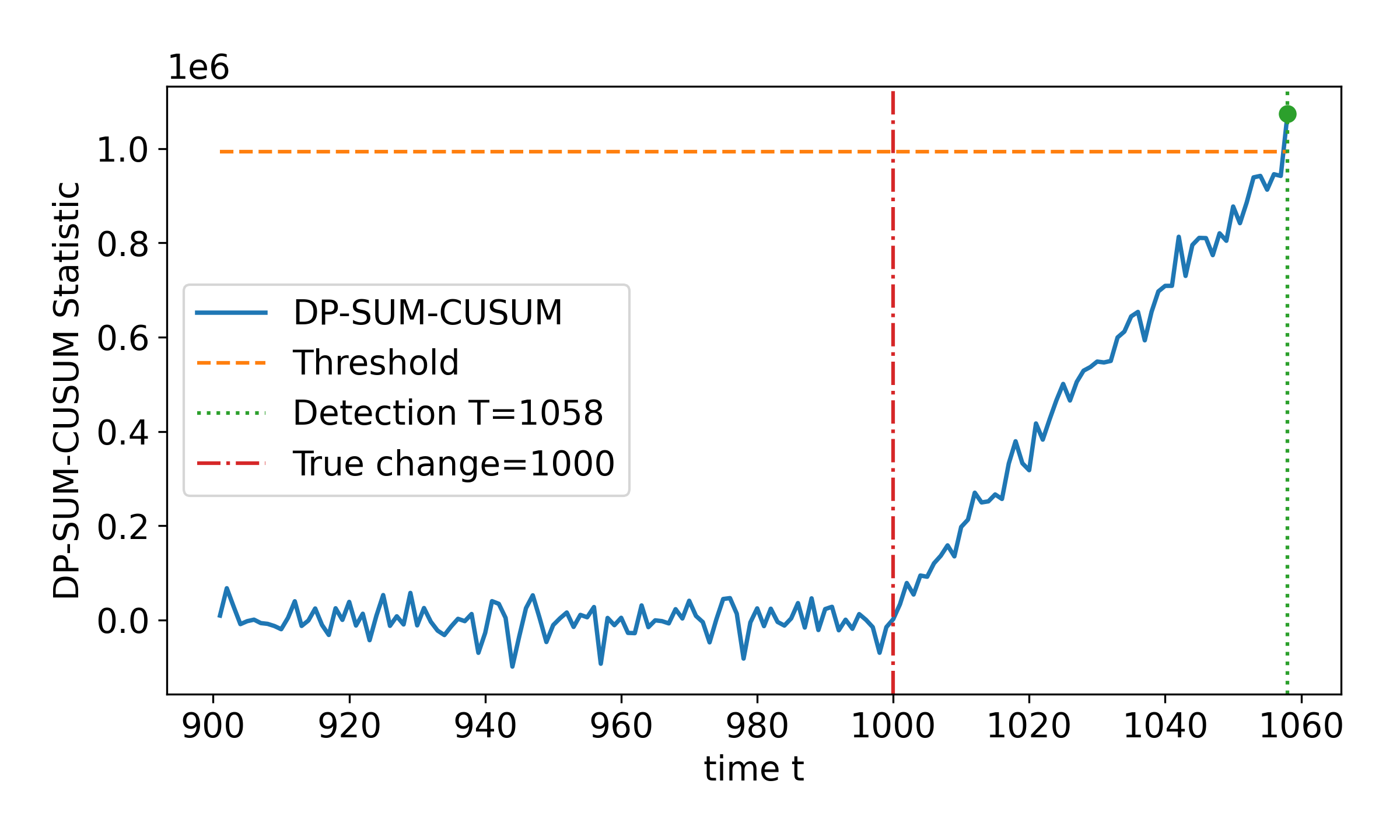}
        \vspace{-0.2in}
        \caption{Trajectory of \algname\ statistics on a real IoT botnet dataset with $K=9$ heterogeneous devices during a \texttt{junk} attack ($\varepsilon=1$). The true change-point (red dash-dotted line) marks the onset of malicious activity.
        }
        \label{fig:real_t1}
\end{figure}

Fig.~\ref{fig:real_t1} illustrates a detection trajectory
for the \texttt{junk} attack under privacy parameter $\varepsilon=1$.
The true change-point corresponds to the onset of attack activity,
after which the aggregated \algname\ statistic exhibits a clear upward trend.
Despite the injected Laplace noise required for privacy preservation,
the statistic crosses the detection threshold shortly after the true
change-point, resulting in a small detection delay. 
These results demonstrate that the proposed privacy-preserving
multi-stream detection procedure remains effective in practice while maintaining differential privacy
guarantees. 

\section{Conclusion}
\label{sec:conclusion}
We studied change-point detection for multi-stream data under differential privacy constraints. We proposed \algname, a privacy-preserving multi-stream detection procedure based on sum-type CUSUM statistics, and derived theoretical guarantees on the average run length and the worst-case average detection delay, characterizing the fundamental tradeoff between privacy and detection efficiency. Future work includes extending the method and analysis to enable identification of the data streams undergoing change, as well as improving robustness via sum-shrinkage schemes, particularly in regimes where only a small and unknown subset of streams changes among a large number of monitored streams.

\bibliographystyle{IEEEtran}
\bibliography{ref}

\clearpage

\onecolumn

\end{document}